\documentclass{article}
\usepackage{amssymb}
\usepackage{amsmath}

\newtheorem{theorem}{Theorem}

\newtheorem{corollary}{Corollary}

\newtheorem{lemma}{Lemma}

\newtheorem{remark}{Remark}

\newenvironment{proof}[1][Proof]{\noindent\textbf{#1.} }{\ \rule{0.5em}{0.5em}}
\input{tcilatex}

\begin{document}

\title{\textbf{Approximation of conjugate functions\ by general linear
operators of their Fourier series at the Lebesgue points}}
\author{\textbf{W\l odzimierz \L enski \ and Bogdan Szal} \\
University of Zielona G\'{o}ra\\
Faculty of Mathematics, Computer Science and Econometrics\\
65-516 Zielona G\'{o}ra, ul. Szafrana 4a, Poland\\
W.Lenski@wmie.uz.zgora.pl, \\
B.Szal @wmie.uz.zgora.pl}
\date{}
\maketitle

\begin{abstract}
The pointwise estimates of the deviations $\widetilde{T}_{n,A,B}^{\text{ }%
}f\left( \cdot \right) -\widetilde{f}(\cdot )$ and $\widetilde{T}_{n,A,B}^{%
\text{ }}f\left( \cdot \right) -\widetilde{f}(\cdot ,\varepsilon )$ in terms
of moduli of continuity $\widetilde{\bar{w}}_{\cdot }f$ and $\widetilde{w}%
_{\cdot }f$ are proved. Analogical results on norm approximation with
remarks and corollary are also given. These results generalized a theorem of
Mittal \cite[Theorem 1, p. 437]{M}

\ \ \ \ \ \ \ \ \ \ \ \ \ \ \ \ \ \ \ \ 

\textbf{Key words: }Rate of approximation, summability of Fourier series,

\ \ \ \ \ \ \ \ \ \ \ \ \ \ \ \ \ \ \ 

\textbf{2000 Mathematics Subject Classification: }42A24.
\end{abstract}

\section{Introduction}

Let $L^{p}\ (1\leq p<\infty )\;\left[ p=\infty \right] $ be the class of all 
$2\pi $--periodic real--valued functions integrable in the Lebesgue sense
with $p$--th power $\left[ \text{essentially bounded}\right] $ over $Q=$ $%
[-\pi ,\pi ]$ with the norm%
\begin{equation*}
\Vert f\Vert :=\Vert f(\cdot )\Vert _{_{L^{p}}}=\left\{ 
\begin{array}{c}
\left( \int_{_{_{Q}}}\mid f(t)\mid ^{p}dt\right) ^{1/p}\text{ \ \ when \ }%
1\leq p<\infty , \\ 
ess\sup\limits_{t\in Q}\mid f(t)\mid \text{ \ when \ }p=\infty 
\end{array}%
\right. 
\end{equation*}%
and consider the trigonometric Fourier series 
\begin{equation*}
Sf(x):=\frac{a_{0}(f)}{2}+\sum_{\nu =1}^{\infty }(a_{\nu }(f)\cos \nu
x+b_{\nu }(f)\sin \nu x)
\end{equation*}%
with the partial sums\ $S_{k}f$ and the conjugate one 
\begin{equation*}
\widetilde{S}f(x):=\sum_{\nu =1}^{\infty }(a_{\nu }(f)\sin \nu x-b_{\nu
}(f)\cos \nu x)
\end{equation*}%
with the partial sums $\widetilde{S}_{k}f$. We know that if $f\in L^{1}$
then 
\begin{equation*}
\widetilde{f}\left( x\right) :=-\frac{1}{\pi }\int_{0}^{\pi }\psi _{x}\left(
t\right) \frac{1}{2}\cot \frac{t}{2}dt=\lim_{\epsilon \rightarrow 0^{+}}%
\widetilde{f}\left( x,\epsilon \right) ,
\end{equation*}%
where 
\begin{equation*}
\widetilde{f}\left( x,\epsilon \right) :=-\frac{1}{\pi }\int_{\epsilon
}^{\pi }\psi _{x}\left( t\right) \frac{1}{2}\cot \frac{t}{2}dt
\end{equation*}%
with 
\begin{equation*}
\psi _{x}\left( t\right) :=f\left( x+t\right) -f\left( x-t\right) ,
\end{equation*}%
exists for almost all \ $x$ \cite[Th.(3.1)IV]{Z}.

Let $A:=\left( a_{n,k}\right) $ and $B:=\left( b_{n,k}\right) $ be infinite
lower triangular matrices of real numbers such that%
\begin{eqnarray*}
a_{n,k} &\geq &0\text{ and }b_{n,k}\geq 0\text{ \ when \ }k=0,1,2,...n\text{%
, \ } \\
a_{n,k} &=&0\text{\ and }b_{n,k}=0\text{ \ when }k>n\text{,}
\end{eqnarray*}%
\begin{equation*}
\sum_{k=0}^{n}a_{n,k}=1\text{ \ and \ }\sum_{k=0}^{n}b_{n,k}=1\text{, where }%
n=0,1,2,...\text{.}
\end{equation*}%
Let define the general linear operators by the\ $AB-$transformation$\ $of \ $%
\left( S_{k}f\right) $ and $\left( \widetilde{S}_{k}f\right) $ as follows%
\begin{equation*}
T_{n,A,B}^{\text{ }}f\left( x\right)
:=\sum\limits_{r=0}^{n}\sum\limits_{k=0}^{r}a_{n,r}b_{r,k}S_{k}f\left(
x\right) \text{\ \ \ }\left( n=0,1,2,...\right)
\end{equation*}%
and%
\begin{equation}
\widetilde{T}_{n,A,B}^{\text{ }}f\left( x\right)
:=\sum\limits_{r=0}^{n}\sum\limits_{k=0}^{r}a_{n,r}b_{r,k}\widetilde{S}%
_{k}f\left( x\right) \text{\ \ \ }\left( n=0,1,2,...\right) .  \label{61}
\end{equation}

As a measure of approximation of $f,\widetilde{f}$ and $\widetilde{f}\left(
\cdot ,\epsilon \right) $ by the above quantities we use the pointwise
moduli of continuity of $f$ in the space $L^{1}$ defined by the formulas%
\begin{equation*}
w_{x}f\left( \delta \right) =\frac{1}{\delta }\int_{0}^{\delta }\left\vert
\varphi _{x}\left( u\right) \right\vert du\text{ and }\widetilde{w}%
_{x}f\left( \delta \right) =\frac{1}{\delta }\int_{0}^{\delta }\left\vert
\psi _{x}\left( u\right) \right\vert du,\text{ }
\end{equation*}%
\begin{equation*}
\bar{w}_{x}f\left( \delta \right) =\sup_{0<t\leq \delta }\frac{1}{t}%
\int_{0}^{t}\left\vert \varphi _{x}\left( u\right) \right\vert du\text{ and }%
\widetilde{\overline{w}}_{x}f\left( \delta \right) =\sup_{0<t\leq \delta }%
\frac{1}{t}\int_{0}^{t}\left\vert \psi _{x}\left( u\right) \right\vert du,
\end{equation*}%
where%
\begin{equation*}
\varphi _{x}\left( t\right) :=f\left( x+t\right) +f\left( x-t\right)
-2f\left( x\right) .
\end{equation*}

It is clear that 
\begin{equation}
\Vert w_{\cdot }f\left( \delta \right) \Vert _{_{L^{p}}}\leq \omega f\left(
\delta \right) _{L^{p}}\text{ and }\Vert \widetilde{w}_{\cdot }f\left(
\delta \right) \Vert _{_{L^{p}}}\leq \widetilde{\omega }f\left( \delta
\right) _{L^{p}},  \label{81}
\end{equation}%
where 
\begin{equation*}
\omega f\left( \delta \right) _{L^{p}}=\sup_{0<t\leq \delta }\Vert \varphi
_{\cdot }\left( t\right) \Vert _{_{L^{p}}}\text{ and }\widetilde{\omega }%
f\left( \delta \right) _{L^{p}}=\sup_{0<t\leq \delta }\Vert \psi _{\cdot
}\left( t\right) \Vert _{_{L^{p}}}
\end{equation*}%
are the classical moduli of continuity of $f$.

The deviation $T_{n,A,B}^{\text{ }}f-f$ with lower triangular infinite
matrix $A$, defined by $a_{n,k}=\frac{1}{n+1}$ \ when \ $k=0,1,2,...n$ and $%
a_{n,k}=0$ \ when $k>n$, was estimated by M. L. Mittal as follows:

\noindent \textbf{Theorem A. }\cite[Theorem 1, p. 437]{M} \textit{If entries
of matrix }$A$ \textit{satisfy the conditions}

\begin{equation*}
a_{n,n-k}-a_{n+1,n+1-k}\geq 0\text{\ \ \textit{for} \ }0\leq k\leq n,
\end{equation*}%
\begin{equation*}
\sum_{k=0}^{l}\text{ }\left( k+1\right) \left\vert
a_{r,r-k}-a_{r,r-k-1}\right\vert =O\left( \sum_{k=r-l}^{r}a_{r,k}\right) 
\text{ \ \ \textit{for} \ }0\leq l\leq r\leq n,
\end{equation*}%
\begin{eqnarray*}
&&\sum_{k=0}^{r}\text{ }\left( k+1\right) \left\vert \left(
a_{r,r-k}-a_{r+1,r+1-k}\right) -\left( a_{r,r-k-1}-a_{r+1,r-k}\right)
\right\vert \\
&=&O\left( \frac{1}{r+1}\right) \text{ \ \ \ \ \textit{for} \ }0\leq r\leq n
\end{eqnarray*}%
\textit{and}%
\begin{equation*}
a_{n,0}-a_{n+1,1}=O\left( \left( n+1\right) ^{-2}\right)
\end{equation*}%
\textit{then, for} $x$ \ \textit{such that} 
\begin{equation*}
\frac{1}{t}\int_{0}^{t}\left\vert \varphi _{x}\left( u\right) \right\vert
du=o_{x}\left( 1\right) \text{ \ \ \textit{as }\ }t\rightarrow 0+\text{ ,}
\end{equation*}%
\textit{we have the relation}%
\begin{equation*}
\frac{1}{n+1}\sum\limits_{r=0}^{n}\sum\limits_{k=0}^{r}a_{r,k}S_{k}f\left(
x\right) -f(x)=o_{x}\left( 1\right) \text{ \ \ \textit{as }\ }n\rightarrow
\infty .
\end{equation*}

The deviation $T_{n,A,B}^{\text{ }}f-f$ was also estimated in our earlier
paper as follows:

\noindent \textbf{Theorem B. }\cite{LSz}\textit{\ Let \ }$f\in L^{1}$\textit{%
. If entries of our matrices satisfy the conditions} 
\begin{equation*}
a_{n,n}\ll \frac{1}{n+1},
\end{equation*}%
\begin{equation*}
\frac{1}{s+1}\sum\limits_{r=0}^{s}a_{n,r}\ll a_{n,s}\text{ \ \textit{for} \ }%
0\leq s\leq n
\end{equation*}%
\textit{and}%
\begin{equation*}
\left\vert a_{n,r}b_{r,r-l}-a_{n,r+1}b_{r+1,r+1-l}\right\vert \ll \frac{%
a_{n,r}}{\left( r+1\right) ^{2}}\text{ for }0\leq l\leq r\leq n-1,
\end{equation*}%
\textit{then}%
\begin{equation*}
\left\vert T_{n,A,B}^{\text{ }}f\left( x\right) -f\left( x\right)
\right\vert \ll \sum\limits_{r=0}^{n}a_{n,r}\left[ \frac{1}{r+1}%
\sum\limits_{k=0}^{r}\bar{w}_{x}f\left( \frac{\pi }{k+1}\right) \right] ,
\end{equation*}%
\begin{equation*}
\left\Vert T_{n,A,B}^{\text{ }}f\left( \cdot \right) -f\left( \cdot \right)
\right\Vert _{L^{p}}\ll \sum\limits_{r=0}^{n}a_{n,r}\left[ \frac{1}{r+1}%
\sum\limits_{k=0}^{r}\omega f\left( \frac{\pi }{k+1}\right) _{L^{p}}\right] ,
\end{equation*}%
\textit{for every natural }$n$\textit{\ and all real\ }$x.$

In the case of conjugate functions the deviation $\sum%
\limits_{k=0}^{n}a_{n,k}\widetilde{S}_{k}f\left( x\right) -\widetilde{f}(x)$
was considered by M. L. Mittal, B. E. Rhoades and V. N. Mishra in \cite{MRM}
in the following way

\noindent \textbf{Theorem C. }\cite[Theorem 3.1]{MRM}\textit{\ Let \ }$f\in
L^{p}\left( p\geq 1\right) $\textit{\ such that}%
\begin{equation*}
\int_{0}^{2\pi }\left\vert \left[ f\left( x+t\right) -f\left( x\right) %
\right] \sin ^{\beta }x\right\vert ^{p}dx=O\left( \xi \left( t\right)
\right) \text{ \ \ }\left( \beta \geq 0\right) \text{ }
\end{equation*}%
\textit{provided that }$\xi \left( t\right) $\textit{\ is positive,
increasing function of }$t$\textit{\ satisfying the conditions}%
\begin{equation*}
\left\{ \int_{0}^{\pi /n}\left( \frac{t\left\vert \psi _{x}\left( t\right)
\right\vert }{\xi \left( t\right) }\right) ^{p}\sin ^{\beta p}tdt\right\}
^{1/p}=O_{x}\left( \frac{1}{n}\right)
\end{equation*}%
\begin{equation*}
\left\{ \int_{\pi /n}^{\pi }\left( \frac{t^{-\delta }\left\vert \psi
_{x}\left( t\right) \right\vert }{\xi \left( t\right) }\right)
^{p}dt\right\} ^{1/p}=O_{x}\left( n^{\delta }\right)
\end{equation*}%
\textit{uniformly in }$x$\textit{, where }$\delta <1/p,$\textit{\ but the
function }$\frac{\xi \left( t\right) }{t}$\textit{\ decreasing in }$t.$%
\textit{\ If the entries of matrix }$A$\textit{\ satisfy the condition}%
\begin{equation*}
\sum_{k=0}^{l}\text{ }\left( k+1\right) \left\vert
a_{r,r-k}-a_{r,r-k-1}\right\vert =O\left( \sum_{k=r-l}^{r}a_{r,k}\right) 
\text{ \ \ \textit{for} \ }0\leq l\leq r\leq n,
\end{equation*}%
\textit{then}%
\begin{equation*}
\left\Vert \sum\limits_{k=0}^{n}a_{n,k}\widetilde{S}_{k}f\left( \cdot
\right) -\widetilde{f}(\cdot )\right\Vert _{L^{p}}=O\left( n^{\beta +\frac{1%
}{p}}\xi \left( \frac{1}{n}\right) \right) .
\end{equation*}

In our theorems we will consider the deviations $\ \widetilde{T}_{n,A,B}^{%
\text{ }}f\left( \cdot \right) -\widetilde{f}(\cdot )$ and $\widetilde{T}%
_{n,A,B}^{\text{ }}f\left( \cdot \right) -\widetilde{f}(\cdot ,\varepsilon )$
with the mean (\ref{61}) introduced at the begin and we will present the
estimates of the above type. Consequently, we also give some results on norm
approximation and some remarks. Finally, we will derive a corollary.

We shall write $I_{1}\ll I_{2}$ if there exists a positive constant $K$,
sometimes depend on some parameters, such that $I_{1}\leq KI_{2}$.

\section{Statement of the results}

We start with our main results on the degrees of pointwise summability.

\begin{theorem}
Let \ $f\in L^{1}$. If entries of our matrices satisfy the conditions%
\begin{equation}
a_{n,n}\ll \frac{1}{n+1},  \label{2.1}
\end{equation}%
\begin{equation}
\frac{1}{s+1}\sum\limits_{r=0}^{s}a_{n,r}\ll a_{n,s}\text{ \ \textit{for} \ }%
0\leq s\leq n  \label{2.2}
\end{equation}%
and%
\begin{equation}
\left\vert a_{n,r}b_{r,r-l}-a_{n,r+1}b_{r+1,r+1-l}\right\vert \ll \frac{%
a_{n,r}}{\left( r+1\right) ^{2}}\text{ for }0\leq l\leq r\leq n-1,
\label{2.21}
\end{equation}%
then \ 
\begin{equation}
\left\vert \widetilde{T}_{n,A,B}^{\text{ }}f\left( x\right) -\widetilde{f}%
\left( x,\frac{\pi }{n+1}\right) \right\vert \ll \sum\limits_{r=0}^{n}a_{n,r}%
\left[ \frac{1}{r+1}\sum\limits_{k=0}^{r}\widetilde{\bar{w}}_{x}f\left( 
\frac{\pi }{k+1}\right) \right]  \label{2.51}
\end{equation}%
and under the additional condition 
\begin{equation}
\frac{1}{\pi }\int_{0}^{\frac{\pi }{n+1}}\frac{\left\vert \psi _{x}\left(
t\right) \right\vert }{t}dt\ll \widetilde{w}_{x}f\left( \frac{\pi }{n+1}%
\right) ,  \label{2.511}
\end{equation}%
\begin{equation}
\left\vert \widetilde{T}_{n,A,B}^{\text{ }}f\left( x\right) -\widetilde{f}%
\left( x\right) \right\vert \ll \sum\limits_{r=0}^{n}a_{n,r}\left[ \frac{1}{%
r+1}\sum\limits_{k=0}^{r}\widetilde{\bar{w}}_{x}f\left( \frac{\pi }{k+1}%
\right) \right]  \label{2.5}
\end{equation}%
for every natural $n$ and all considered real\ $x.$
\end{theorem}

\begin{remark}
We can observe that the proof of Theorem 1 yields the following more precise
estimate%
\begin{equation}
\left. 
\begin{array}{c}
\left\vert \widetilde{T}_{n,A,B}^{\text{ }}f\left( x\right) -\widetilde{f}%
\left( x,\frac{\pi }{n+1}\right) \right\vert \\ 
\left\vert \widetilde{T}_{n,A,B}^{\text{ }}f\left( x\right) -\widetilde{f}%
\left( x\right) \right\vert%
\end{array}%
\right\} \ll \sum\limits_{r=0}^{n}\left( a_{n,r}+\sum_{k=1}^{r}\frac{a_{n,k}%
}{r+1}\right) \left[ \frac{1}{r+1}\sum\limits_{k=0}^{r}\widetilde{w}%
_{x}f\left( \frac{\pi }{k+1}\right) \right]  \notag
\end{equation}%
\begin{equation}
+\left[ \frac{1}{n+1}\sum\limits_{k=0}^{n}\widetilde{w}_{x}f\left( \frac{\pi 
}{k+1}\right) \right] ,  \label{2.6}
\end{equation}%
without assumption (\ref{2.2}). In this case we can obtain the relation%
\begin{equation*}
\left. 
\begin{array}{c}
\left\vert \widetilde{T}_{n,A,B}^{\text{ }}f\left( x\right) -\widetilde{f}%
\left( x,\frac{\pi }{n+1}\right) \right\vert \\ 
\left\vert \widetilde{T}_{n,A,B}^{\text{ }}f\left( x\right) -\widetilde{f}%
\left( x\right) \right\vert%
\end{array}%
\right\} =o_{x}\left( 1\right)
\end{equation*}%
under weaker assumption $\sum_{r=0}^{n}\sum_{k=0}^{r}\frac{a_{n,k}}{r+1}%
=O\left( 1\right) $ instead of (\ref{2.2}) with (\ref{2.1}), (\ref{2.21})
and $\widetilde{w}_{x}f\left( \delta \right) =o_{x}\left( 1\right) .$
\end{remark}

If we suppose at the beginning that the matrix\ $A$ is such that $a_{n,k}=%
\frac{1}{n+1}$ when $k=0,1,2,...n$ and $a_{n,k}=0$ when $k>n,$ then we can
yet reduce our assumptions.

\begin{theorem}
Let \ $f\in L^{1}$ and let entries of the matrix $B$ satisfy the condition 
\begin{equation}
\left\vert b_{r,r-l}-b_{r+1,r+1-l}\right\vert \ll \frac{1}{\left( r+1\right)
^{2}}\text{ for }0\leq l\leq r.  \label{3.2}
\end{equation}%
Then%
\begin{equation*}
\left\vert \widetilde{T}_{n,\left( \frac{1}{n+1}\right) ,B}^{\text{ }%
}f\left( x\right) -\widetilde{f}\left( x,\frac{\pi }{n+1}\right) \right\vert
\ll \frac{1}{n+1}\text{ }\sum\limits_{r=0}^{n}\left[ \frac{1}{r+1}%
\sum\limits_{k=0}^{r}\widetilde{w}_{x}f\left( \frac{\pi }{k+1}\right) \right]
\end{equation*}%
and under (\ref{2.511}) 
\begin{equation}
\left\vert \widetilde{T}_{n,\left( \frac{1}{n+1}\right) ,B}^{\text{ }%
}f\left( x\right) -\widetilde{f}\left( x\right) \right\vert \ll \frac{1}{n+1}%
\text{ }\sum\limits_{r=0}^{n}\left[ \frac{1}{r+1}\sum\limits_{k=0}^{r}%
\widetilde{w}_{x}f\left( \frac{\pi }{k+1}\right) \right] ,  \label{3.4}
\end{equation}%
for every natural $n$ and all considered real\ $x.$
\end{theorem}

\begin{remark}
Analyzing the proof of Theorem 1 with $A=\left( \frac{1}{n+1}\right) $ we
can see that under the following additional assumption on the entries of the
matrix $B$%
\begin{equation*}
\sum\limits_{r=s}^{n-1}\sum\limits_{k=s}^{r}\left\vert
b_{r,r-k}-b_{r+1,r+1-k}\right\vert \ll 1\text{ \ for }0<s\leq r\leq n-1,
\end{equation*}%
we obtain the more precise estimate 
\begin{equation*}
\left\vert \widetilde{T}_{n,\left( \frac{1}{n+1}\right) ,B}^{\text{ }%
}f\left( x\right) -\widetilde{f}\left( x\right) \right\vert \ll \frac{1}{n+1}%
\text{ }\sum\limits_{r=0}^{n}\widetilde{w}_{x}f\left( \frac{\pi }{r+1}\right)
\end{equation*}%
than (\ref{3.4}). This additional assumption as well (\ref{3.2}) are
fulfilled if $b_{n,k}=0$ for $k=0,1,2,...,n-1,n+1,...$ and $b_{n,n}=1.$ Thus
we have analogue of the well known classical result of S. Aljan\v{c}i\v{c},
R. Bojanic and M. Tomi\'{c} \cite{ABT}. The analogical remark we can prepare
with respect to Theorem 1.
\end{remark}

Finally, we formulate the results on the estimates of $L^{p}$norm of the
deviation considered above.

\begin{theorem}
Let \ $f\in L^{p}$. Under the assumptions of Theorem 1 on the entries of
matrices $A$ and $B$ and $\widetilde{\omega }f$, we have 
\begin{equation*}
\left. 
\begin{array}{c}
\left\Vert \widetilde{T}_{n,A,B}^{\text{ }}f\left( \cdot \right) -\widetilde{%
f}\left( \cdot ,\frac{\pi }{n+1}\right) \right\Vert _{L^{p}} \\ 
\left\Vert \widetilde{T}_{n,A,B}^{\text{ }}f\left( \cdot \right) -\widetilde{%
f}\left( \cdot \right) \right\Vert _{L^{p}}%
\end{array}%
\right\} \ll \sum\limits_{r=0}^{n}a_{n,r}\left[ \frac{1}{r+1}%
\sum\limits_{k=0}^{r}\widetilde{\omega }f\left( \frac{\pi }{k+1}\right)
_{L^{p}}\right] ,
\end{equation*}%
for every natural $n.$
\end{theorem}

\begin{theorem}
Let \ $f\in L^{p}$ ,\ then, under the assumptions of Theorem 2 on the
entries of matrix $B$ and $\widetilde{\omega }f$, we have 
\begin{equation*}
\left. 
\begin{array}{c}
\left\Vert \widetilde{T}_{n,\left( \frac{1}{n+1}\right) ,B}^{\text{ }%
}f\left( \cdot \right) -\widetilde{f}\left( \cdot ,\frac{\pi }{n+1}\right)
\right\Vert _{L^{p}} \\ 
\left\Vert \widetilde{T}_{n,\left( \frac{1}{n+1}\right) ,B}^{\text{ }%
}f\left( \cdot \right) -\widetilde{f}\left( \cdot \right) \right\Vert
_{L^{p}}%
\end{array}%
\right\} \ll \frac{1}{n+1}\sum\limits_{r=0}^{n}\left[ \frac{1}{r+1}%
\sum\limits_{k=0}^{r}\widetilde{\omega }f\left( \frac{\pi }{k+1}\right)
_{L^{p}}\right] ,
\end{equation*}%
for every natural $n.$
\end{theorem}

\begin{corollary}
Under the assumptions of our Theorems we have the relations%
\begin{eqnarray*}
\left. 
\begin{array}{c}
\left\vert \widetilde{T}_{n,A,B}^{\text{ }}f\left( x\right) -\widetilde{f}%
\left( x,\frac{\pi }{n+1}\right) \right\vert \\ 
\left\vert \widetilde{T}_{n,A,B}^{\text{ }}f\left( x\right) -\widetilde{f}%
\left( x\right) \right\vert%
\end{array}%
\right\} &=&o_{x}\left( 1\right) ,\text{ \ a.e. in \ }x, \\
\left. 
\begin{array}{c}
\left\vert \widetilde{T}_{n,\left( \frac{1}{n+1}\right) ,B}^{\text{ }%
}f\left( x\right) -\widetilde{f}\left( x,\frac{\pi }{n+1}\right) \right\vert
\\ 
\left\vert \widetilde{T}_{n,\left( \frac{1}{n+1}\right) ,B}^{\text{ }%
}f\left( x\right) -\widetilde{f}\left( x\right) \right\vert%
\end{array}%
\right\} &=&o_{x}\left( 1\right) ,\text{\ a.e. in \ }x,
\end{eqnarray*}%
such that $\widetilde{w}_{x}f\left( t\right) =o_{x}\left( 1\right) $ \ \ 
\textit{as }\ $t\rightarrow 0+,$ and 
\begin{eqnarray*}
\left. 
\begin{array}{c}
\left\Vert \widetilde{T}_{n,A,B}^{\text{ }}f\left( \cdot \right) -\widetilde{%
f}\left( \cdot ,\frac{\pi }{n+1}\right) \right\Vert _{L^{p}} \\ 
\left\Vert \widetilde{T}_{n,A,B}^{\text{ }}f\left( \cdot \right) -\widetilde{%
f}\left( \cdot \right) \right\Vert _{L^{p}}%
\end{array}%
\right\} &=&o\left( 1\right) , \\
\left. 
\begin{array}{c}
\left\Vert \widetilde{T}_{n,\left( \frac{1}{n+1}\right) ,B}^{\text{ }%
}f\left( \cdot \right) -\widetilde{f}\left( \cdot ,\frac{\pi }{n+1}\right)
\right\Vert _{L^{p}} \\ 
\left\Vert \widetilde{T}_{n,\left( \frac{1}{n+1}\right) ,B}^{\text{ }%
}f\left( \cdot \right) -\widetilde{f}\left( \cdot \right) \right\Vert
_{L^{p}}%
\end{array}%
\right\} &=&o\left( 1\right) ,
\end{eqnarray*}%
when $n\rightarrow \infty .$
\end{corollary}

\section{Auxiliary results}

We begin\ this section by some notations following A. Zygmund \cite[Section
5 of Chapter II]{Z}.

It is clear that 
\begin{equation*}
\widetilde{S}_{k}f\left( x\right) =-\frac{1}{\pi }\int_{-\pi }^{\pi }f\left(
x+t\right) \widetilde{D}_{k}\left( t\right) dt,
\end{equation*}%
and 
\begin{equation*}
\widetilde{T}_{n,A,B}^{\text{ }}f\left( x\right) =-\frac{1}{\pi }\int_{-\pi
}^{\pi }f\left( x+t\right)
\sum\limits_{r=0}^{n}\sum\limits_{k=0}^{r}a_{n,r}b_{r,k}\widetilde{D_{k}}%
\left( t\right) dt,
\end{equation*}%
where 
\begin{equation*}
\text{\ }\widetilde{D_{k}}\left( t\right) =\sum_{\nu =0}^{k}\sin \nu t=\frac{%
\cos \frac{t}{2}-\cos \frac{\left( 2k+1\right) t}{2}}{2\sin \frac{t}{2}}=%
\frac{2\sin \frac{kt}{2}\sin \frac{\left( k+1\right) t}{2}}{2\sin \frac{t}{2}%
}.
\end{equation*}%
Hence 
\begin{eqnarray*}
\widetilde{T}_{n,A,B}^{\text{ }}f\left( x\right) -\widetilde{f}\left( x,%
\frac{\pi }{n+1}\right) &=&-\frac{1}{\pi }\int_{0}^{\frac{\pi }{n+1}}\psi
_{x}\left( t\right) \sum\limits_{r=0}^{n}\sum\limits_{k=0}^{r}a_{n,r}b_{r,k}%
\widetilde{D_{k}}\left( t\right) dt \\
&&+\frac{1}{\pi }\int_{\frac{\pi }{n+1}}^{\pi }\psi _{x}\left( t\right)
\sum\limits_{r=0}^{n}\sum\limits_{k=0}^{r}a_{n,r}b_{r,k}\widetilde{%
D_{k}^{\circ }}\left( t\right) dt
\end{eqnarray*}%
and 
\begin{equation*}
\widetilde{T}_{n,A,B}^{\text{ }}f\left( x\right) -\widetilde{f}\left(
x\right) =\frac{1}{\pi }\int_{0}^{\pi }\psi _{x}\left( t\right)
\sum\limits_{r=0}^{n}\sum\limits_{k=0}^{r}a_{n,r}b_{r,k}\widetilde{%
D_{k}^{\circ }}\left( t\right) dt,
\end{equation*}%
where 
\begin{equation*}
\widetilde{D_{k}^{\circ }}\left( t\right) =\frac{1}{2}\cot \frac{t}{2}%
-\sum_{\nu =0}^{k}\sin \nu t=\frac{\cos \frac{\left( 2k+1\right) t}{2}}{%
2\sin \frac{t}{2}}.
\end{equation*}%
Now, we formulate some estimates for the conjugate Dirichlet kernels.

\begin{lemma}
(see \cite[Section 5 of Chapter II, p. 51]{Z}) If \ $0<\left\vert
t\right\vert \leq \pi /2,$ then 
\begin{equation*}
\left\vert \widetilde{D_{k}^{\circ }}\left( t\right) \right\vert \leq \frac{%
\pi }{2\left\vert t\right\vert }\text{\ \ and \ }\left\vert \widetilde{D_{k}}%
\left( t\right) \right\vert \leq \frac{\pi }{\left\vert t\right\vert }
\end{equation*}%
but for any real\ $t$ we have 
\begin{equation*}
\left\vert \widetilde{D_{k}}\left( t\right) \right\vert \leq \frac{1}{2}%
k\left( k+1\right) \left\vert t\right\vert \text{ \ and \ \ }\left\vert 
\widetilde{D_{k}}\left( t\right) \right\vert \leq k+1.
\end{equation*}
\end{lemma}

More complicated estimates we give with proofs.

\begin{lemma}
Let \ $f\in L^{1}$. The following inequalities 
\begin{eqnarray*}
\widetilde{w}_{x}f\left( \frac{\pi }{n+1}\right)  &\leq &2\frac{1}{n+1}%
\sum_{r=0}^{n}\widetilde{w}_{x}f\left( \frac{\pi }{r+1}\right)  \\
\widetilde{\overline{w}}_{x}f\left( \frac{\pi }{n+1}\right)  &\leq &\frac{1}{%
n+1}\sum_{r=0}^{n}\widetilde{\overline{w}}_{x}f\left( \frac{\pi }{r+1}%
\right) 
\end{eqnarray*}%
hold for every naturals $n$ and all real $x$.
\end{lemma}

\begin{proof}
The proof \ of the first inequality follows by the easy account 
\begin{eqnarray*}
\widetilde{w}_{x}f\left( \frac{\pi }{n+1}\right) &=&\frac{n+1}{\pi }%
\int_{0}^{\frac{\pi }{n+1}}\left\vert \psi _{x}\left( u\right) \right\vert
du\sum_{r=0}^{n}\frac{2\left( r+1\right) }{\left( n+1\right) \left(
n+2\right) } \\
&=&\frac{2}{n+1}\sum_{r=0}^{n}\frac{\left( r+1\right) \left( n+1\right) }{%
\pi \left( n+2\right) }\int_{0}^{\frac{\pi }{n+1}}\left\vert \psi _{x}\left(
u\right) \right\vert du \\
&\leq &\frac{2}{n+1}\sum_{r=0}^{n}\frac{r+1}{\pi }\int_{0}^{\frac{\pi }{r+1}%
}\left\vert \psi _{x}\left( u\right) \right\vert du \\
&=&2\frac{1}{n+1}\sum_{r=0}^{n}\widetilde{w}_{x}f\left( \frac{\pi }{r+1}%
\right) .
\end{eqnarray*}%
The second inequality is evident, therefore our proof is completed.
\end{proof}

\section{Proofs of the results}

\subsection{Proof of Theorem 1}

First, we prove the relation (\ref{2.51}). The main idea of the proof based
on the method used in \cite[Proof of Theorem 1]{LSz}. Let%
\begin{equation*}
\widetilde{T}_{n,A,B}f\left( x\right) -\widetilde{f}\left( x,\frac{\pi }{n+1}%
\right) =\sum\limits_{r=0}^{n}\sum\limits_{k=0}^{r}a_{n,r}b_{r,k}\widetilde{S%
}_{k}\left( x\right) -\widetilde{f}\left( x,\frac{\pi }{n+1}\right)
\end{equation*}%
\begin{eqnarray*}
&=&-\frac{1}{\pi }\int_{0}^{\frac{\pi }{n+1}}\psi _{x}\left( t\right)
\sum\limits_{r=0}^{n}\sum\limits_{k=0}^{r}a_{n,r}b_{r,k}\widetilde{D}%
_{k}\left( t\right) dt+\frac{1}{\pi }\int_{\frac{\pi }{n+1}}^{\pi }\psi
_{x}\left( t\right) \sum\limits_{r=0}^{n}\sum\limits_{k=0}^{r}a_{n,r}b_{r,k}%
\widetilde{D_{k}^{\circ }}\left( t\right) dt \\
&=&I_{1}+I_{2}.
\end{eqnarray*}%
Further, by Lemma 1,%
\begin{equation*}
\left\vert I_{1}\right\vert \leq \frac{1}{\pi }\int_{0}^{\frac{\pi }{n+1}%
}\left\vert \psi _{x}\left( t\right) \right\vert
\sum\limits_{r=0}^{n}\sum\limits_{k=0}^{r}a_{n,r}b_{r,k}\left\vert \frac{%
2\sin \frac{kt}{2}\sin \frac{\left( k+1\right) t}{2}}{2\sin \frac{t}{2}}%
\right\vert dt
\end{equation*}%
\begin{equation*}
\leq \frac{1}{n+1}\widetilde{w}_{x}f\left( \frac{\pi }{n+1}\right)
\sum\limits_{r=0}^{n}\sum\limits_{k=0}^{r}\left( k+1\right) a_{n,r}b_{r,k}
\end{equation*}

\begin{equation*}
\leq \widetilde{w}_{x}f\left( \frac{\pi }{n+1}\right)
\sum\limits_{r=0}^{n}a_{n,r}\leq \sum\limits_{r=0}^{n}a_{n,r}\widetilde{\bar{%
w}}_{x}f\left( \frac{\pi }{r+1}\right) 
\end{equation*}%
\begin{equation*}
\ll \sum\limits_{r=0}^{n}a_{n,r}\left[ \frac{1}{r+1}\sum\limits_{k=0}^{r}%
\widetilde{\bar{w}}_{x}f\left( \frac{\pi }{k+1}\right) \right] 
\end{equation*}%
and using the inequality: $\sin t\leq \frac{2t}{\pi }$ for $0<t\leq \frac{%
\pi }{2}$%
\begin{equation*}
\left\vert I_{2}\right\vert \leq \frac{1}{\pi }\int_{\frac{\pi }{n+1}}^{\pi
}\left\vert \psi _{x}\left( t\right)
\sum\limits_{r=0}^{n}\sum\limits_{k=0}^{r}a_{n,r}b_{r,k}\frac{\cos \frac{%
\left( 2k+1\right) t}{2}}{2\sin \frac{t}{2}}\right\vert dt
\end{equation*}%
\begin{equation*}
\ll \int_{\frac{\pi }{n+1}}^{\pi }\frac{\left\vert \psi _{x}\left( t\right)
\right\vert }{t}\left\vert
\sum\limits_{r=0}^{n}\sum\limits_{k=0}^{r}a_{n,r}b_{r,k}\cos \frac{\left(
2k+1\right) t}{2}\right\vert dt
\end{equation*}%
\begin{equation*}
\leq \int_{\frac{\pi }{n+1}}^{\pi }\frac{\left\vert \psi _{x}\left( t\right)
\right\vert }{t}\left\vert \sum\limits_{r=0}^{\tau
-1}\sum\limits_{k=0}^{r}a_{n,r}b_{r,k}\cos \frac{\left( 2k+1\right) t}{2}%
\right\vert dt
\end{equation*}%
\begin{equation*}
+\int_{\frac{\pi }{n+1}}^{\pi }\frac{\left\vert \psi _{x}\left( t\right)
\right\vert }{t}\left\vert \sum\limits_{r=\tau }^{n}\sum\limits_{k=0}^{\tau
-1}a_{n,r}b_{r,r-k}\cos \frac{\left( 2r-2k+1\right) t}{2}\right\vert dt
\end{equation*}%
\begin{equation*}
+\int_{\frac{\pi }{n+1}}^{\pi }\frac{\left\vert \psi _{x}\left( t\right)
\right\vert }{t}\left\vert \sum\limits_{r=\tau }^{n}\sum\limits_{k=\tau
}^{r}a_{n,r}b_{r,r-k}\cos \frac{\left( 2r-2k+1\right) t}{2}\right\vert dt
\end{equation*}%
\begin{equation*}
=I_{21}+I_{22}+I_{23},
\end{equation*}%
where $\tau =\left[ \frac{\pi }{t}\right] $ for $t\in (0,\pi ].$ Now, we
shall estimate the integrals $I_{21}$, $I_{22}$ and $I_{23}$ So, 
\begin{equation*}
I_{21}\leq \int_{\frac{\pi }{n+1}}^{\pi }\frac{\left\vert \psi _{x}\left(
t\right) \right\vert }{t}\sum\limits_{r=0}^{\tau
-1}\sum\limits_{k=0}^{r}a_{n,r}b_{r,k}dt
\end{equation*}%
\begin{equation*}
=\int_{\frac{\pi }{n+1}}^{\pi }\frac{\left\vert \psi _{x}\left( t\right)
\right\vert }{t}\sum\limits_{r=0}^{\tau -1}a_{n,r}dt\leq 2\pi \int_{\frac{%
\pi }{n+1}}^{\pi }\frac{\left\vert \psi _{x}\left( t\right) \right\vert }{%
t^{2}}\sum\limits_{r=0}^{\tau -1}\frac{a_{n,r}}{\tau +1}dt
\end{equation*}%
\begin{equation*}
=2\pi \sum\limits_{s=1}^{n}\int_{\frac{\pi }{s+1}}^{\frac{\pi }{s}}\frac{%
\left\vert \psi _{x}\left( t\right) \right\vert }{t^{2}}\sum\limits_{r=0}^{%
\tau -1}\frac{a_{n,r}}{\tau +1}dt\leq 2\pi
\sum\limits_{s=1}^{n}\sum\limits_{r=0}^{s}\frac{a_{n,r}}{s+1}\int_{\frac{\pi 
}{s+1}}^{\frac{\pi }{s}}\frac{\left\vert \psi _{x}\left( t\right)
\right\vert }{t^{2}}dt
\end{equation*}%
and integrating by parts%
\begin{eqnarray*}
I_{21} &\leq &2\pi \sum\limits_{s=1}^{n}\sum\limits_{r=0}^{s}\frac{a_{n,r}}{%
s+1}\left\{ \left[ \frac{1}{t^{2}}\int_{0}^{t}\left\vert \psi _{x}\left(
u\right) \right\vert du\right] _{t=\frac{\pi }{s+1}}^{\frac{\pi }{s}}+2\int_{%
\frac{\pi }{s+1}}^{\frac{\pi }{s}}\left[ \frac{1}{t^{3}}\int_{0}^{t}\left%
\vert \psi _{x}\left( u\right) \right\vert du\right] dt\right\}  \\
&=&2\pi \sum\limits_{s=0}^{n-1}\sum\limits_{r=0}^{s}\frac{a_{n,r}}{s+2}\left[
\frac{1}{t^{2}}\int_{0}^{t}\left\vert \psi _{x}\left( u\right) \right\vert du%
\right] _{t=\frac{\pi }{s+2}}^{\frac{\pi }{s+1}}+2\pi \sum\limits_{s=1}^{n}%
\frac{a_{n,s}}{s+1}\left[ \frac{1}{t^{2}}\int_{0}^{t}\left\vert \psi
_{x}\left( u\right) \right\vert du\right] _{t=\frac{\pi }{s+1}}^{\frac{\pi }{%
s}}
\end{eqnarray*}%
\begin{eqnarray*}
&&+4\pi \sum\limits_{s=1}^{n}\sum\limits_{r=0}^{s}\frac{a_{n,r}}{s+1}\int_{%
\frac{\pi }{s+1}}^{\frac{\pi }{s}}\left[ \frac{1}{t^{3}}\int_{0}^{t}\left%
\vert \psi _{x}\left( u\right) \right\vert du\right] dt \\
&=&2\pi \sum\limits_{s=0}^{n-1}\sum\limits_{r=0}^{s}\frac{a_{n,r}}{s+2}\left[
\frac{\widetilde{w}_{x}f\left( \frac{\pi }{s+1}\right) }{\frac{\pi }{s+1}}-%
\frac{\widetilde{w}_{x}f\left( \frac{\pi }{s+2}\right) }{\frac{\pi }{s+2}}%
\right] 
\end{eqnarray*}%
\begin{eqnarray*}
&&+2\pi \sum\limits_{s=1}^{n}\frac{a_{n,s}}{s+1}\left[ \frac{\widetilde{w}%
_{x}f\left( \frac{\pi }{s}\right) }{\frac{\pi }{s}}-\frac{\widetilde{w}%
_{x}f\left( \frac{\pi }{s+1}\right) }{\frac{\pi }{s+1}}\right]  \\
&&+4\pi \sum\limits_{s=1}^{n}\sum\limits_{r=0}^{s}\frac{a_{n,r}}{s+1}\int_{%
\frac{\pi }{s+1}}^{\frac{\pi }{s}}\left[ \frac{1}{t^{3}}\int_{0}^{t}\left%
\vert \psi _{x}\left( u\right) \right\vert du\right] dt
\end{eqnarray*}%
\begin{eqnarray*}
&\leq &2\pi \sum\limits_{s=0}^{n-1}\sum\limits_{r=0}^{s}\frac{a_{n,r}}{r+1}%
\left[ \frac{\widetilde{w}_{x}f\left( \frac{\pi }{s+1}\right) }{\frac{\pi }{%
s+1}}-\frac{\widetilde{w}_{x}f\left( \frac{\pi }{s+2}\right) }{\frac{\pi }{%
s+2}}\right] +2\pi \sum\limits_{s=1}^{n}\frac{a_{n,s}}{s+1}\frac{\widetilde{w%
}_{x}f\left( \frac{\pi }{s}\right) }{\frac{\pi }{s}} \\
&&+4\pi \sum\limits_{s=1}^{n}\sum\limits_{r=0}^{s}\frac{a_{n,r}}{s+1}\int_{%
\frac{\pi }{s+1}}^{\frac{\pi }{s}}\left[ \frac{1}{t^{3}}\int_{0}^{t}\left%
\vert \psi _{x}\left( u\right) \right\vert du\right] dt.
\end{eqnarray*}%
Changing the order of summation we get%
\begin{eqnarray*}
I_{21} &\ll &\sum\limits_{r=0}^{n-1}\frac{a_{n,r}}{r+1}\sum%
\limits_{s=r}^{n-1}\left[ \frac{\widetilde{w}_{x}f\left( \frac{\pi }{s+1}%
\right) }{\frac{\pi }{s+1}}-\frac{\widetilde{w}_{x}f\left( \frac{\pi }{s+2}%
\right) }{\frac{\pi }{s+2}}\right] +\sum\limits_{s=1}^{n}a_{n,s}\widetilde{w}%
_{x}f\left( \frac{\pi }{s}\right)  \\
&&+\sum\limits_{s=1}^{n}\sum\limits_{r=0}^{s}\frac{a_{n,r}}{s+1}\widetilde{w}%
_{x}f\left( \frac{\pi }{s}\right) \int_{\frac{\pi }{s+1}}^{\frac{\pi }{s}}%
\frac{1}{t^{2}}dt
\end{eqnarray*}%
\begin{eqnarray*}
&=&\sum\limits_{r=0}^{n-1}\frac{a_{n,r}}{r+1}\left[ \frac{\widetilde{w}%
_{x}f\left( \frac{\pi }{r+1}\right) }{\frac{\pi }{r+1}}-\frac{\widetilde{w}%
_{x}f\left( \frac{\pi }{n+1}\right) }{\frac{\pi }{n+1}}\right]
+\sum\limits_{s=1}^{n}a_{n,s}\widetilde{w}_{x}f\left( \frac{\pi }{s}\right) 
\\
&&+\sum\limits_{s=1}^{n}\sum\limits_{r=0}^{s}\frac{a_{n,r}}{s+1}\widetilde{w}%
_{x}f\left( \frac{\pi }{s}\right) \left( \frac{1}{\frac{\pi }{s+1}}-\frac{1}{%
\frac{\pi }{s}}\right) 
\end{eqnarray*}%
\begin{equation*}
\ll \sum\limits_{r=0}^{n-1}\frac{a_{n,r}}{r+1}\frac{\widetilde{w}_{x}f\left( 
\frac{\pi }{r+1}\right) }{\frac{\pi }{r+1}}+\sum\limits_{s=1}^{n}a_{n,s}%
\widetilde{w}_{x}f\left( \frac{\pi }{s}\right)
+\sum\limits_{s=1}^{n}\sum\limits_{r=0}^{s}\frac{a_{n,r}}{s+1}\widetilde{w}%
_{x}f\left( \frac{\pi }{s}\right) 
\end{equation*}%
\begin{eqnarray*}
&\leq &\frac{1}{\pi }\sum\limits_{r=0}^{n}a_{n,r}\widetilde{w}_{x}f\left( 
\frac{\pi }{r+1}\right) +\sum\limits_{s=1}^{n}\left(
a_{n,s}+\sum\limits_{r=0}^{s}\frac{a_{n,r}}{s+1}\right) \widetilde{w}%
_{x}f\left( \frac{\pi }{s}\right)  \\
&\ll &\sum\limits_{r=0}^{n}a_{n,r}\left[ \frac{1}{r+1}\sum\limits_{s=0}^{r}%
\widetilde{w}_{x}f\left( \frac{\pi }{s+1}\right) \right]  \\
&&+\sum\limits_{s=1}^{n}\left( a_{n,s}+\sum\limits_{r=0}^{s}\frac{a_{n,r}}{%
s+1}\right) \left[ \frac{1}{s}\sum\limits_{r=1}^{s}\widetilde{w}_{x}f\left( 
\frac{\pi }{r}\right) \right] 
\end{eqnarray*}%
\begin{eqnarray*}
&=&\sum\limits_{r=0}^{n}a_{n,r}\left[ \frac{1}{r+1}\sum\limits_{s=0}^{r}%
\widetilde{w}_{x}f\left( \frac{\pi }{s+1}\right) \right]  \\
&&+\sum\limits_{s=1}^{n}\left( a_{n,s}+\sum\limits_{r=0}^{s}\frac{a_{n,r}}{%
s+1}\right) \left[ \frac{1}{s}\sum\limits_{r=0}^{s-1}\widetilde{w}%
_{x}f\left( \frac{\pi }{r+1}\right) \right]  \\
&\ll &\sum\limits_{r=0}^{n}\left( a_{n,r}+\sum\limits_{s=0}^{r}\frac{a_{n,s}%
}{r+1}\right) \left[ \frac{1}{r+1}\sum\limits_{s=0}^{r}\widetilde{w}%
_{x}f\left( \frac{\pi }{s+1}\right) \right] .
\end{eqnarray*}%
By the assumption (\ref{2.2}) 
\begin{equation*}
I_{21}\ll \sum\limits_{r=0}^{n}a_{n,r}\left[ \frac{1}{r+1}%
\sum\limits_{s=0}^{r}\widetilde{w}_{x}f\left( \frac{\pi }{s+1}\right) \right]
\leq \sum\limits_{r=0}^{n}a_{n,r}\left[ \frac{1}{r+1}\sum\limits_{k=0}^{r}%
\widetilde{\bar{w}}_{x}f\left( \frac{\pi }{k+1}\right) \right] .
\end{equation*}%
Using Abel's transformation in $I_{22}$ and conditions (\ref{2.1}), and (\ref%
{2.21}), we obtain 
\begin{equation*}
I_{22}=\int_{\frac{\pi }{n+1}}^{\pi }\frac{\left\vert \psi _{x}\left(
t\right) \right\vert }{t}\left\vert \sum\limits_{k=0}^{\tau -1}\left(
\sum\limits_{r=\tau }^{n-1}\left[ a_{n,r}b_{r,r-k}-a_{n,r+1}b_{r+1,r+1-k}%
\right] \sum_{l=\tau }^{r}\cos \frac{\left( 2l-2k+1\right) t}{2}\right.
\right. 
\end{equation*}%
\begin{equation*}
+\left. \left. a_{n,n}b_{n,n-k}\sum_{l=\tau }^{n}\cos \frac{\left(
2l-2k+1\right) t}{2}\right) \right\vert dt
\end{equation*}%
\begin{equation*}
\leq \int_{\frac{\pi }{n+1}}^{\pi }\frac{\left\vert \psi _{x}\left( t\right)
\right\vert }{t}\sum\limits_{k=0}^{\tau -1}\left[ \sum\limits_{r=\tau
}^{n-1}\left\vert a_{n,r}b_{r,r-k}-a_{n,r+1}b_{r+1,r+1-k}\right\vert \tau
+a_{n,n}b_{n,n-k}\tau \right] dt
\end{equation*}%
\begin{equation*}
\leq \int_{\frac{\pi }{n+1}}^{\pi }\frac{\left\vert \psi _{x}\left( t\right)
\right\vert }{t}\sum\limits_{k=0}^{\tau }\tau \left[ \sum\limits_{r=\tau
}^{n-1}\frac{a_{n,r}}{\left( r+1\right) ^{2}}+a_{n,n}b_{n,n-k}\right] dt
\end{equation*}%
\begin{equation*}
\leq \int_{\frac{\pi }{n+1}}^{\pi }\frac{\left\vert \psi _{x}\left( t\right)
\right\vert }{t}\tau \left[ \left( \tau +1\right) \sum\limits_{r=\tau }^{n-1}%
\frac{a_{n,r}}{\left( r+1\right) ^{2}}+a_{n,n}\sum\limits_{k=0}^{\tau
}b_{n,n-k}\right] dt
\end{equation*}%
\begin{equation*}
\leq \pi \int_{\frac{\pi }{n+1}}^{\pi }\frac{\left\vert \psi _{x}\left(
t\right) \right\vert }{t^{2}}\left[ \sum\limits_{r=\tau }^{n-1}\frac{a_{n,r}%
}{r+1}+a_{n,n}\right] dt
\end{equation*}%
\begin{equation*}
=\pi \int_{\frac{\pi }{n+1}}^{\pi }\frac{\left\vert \psi _{x}\left( t\right)
\right\vert }{t^{2}}\sum\limits_{r=\tau }^{n-1}\frac{a_{n,r}}{r+1}dt+\pi
a_{n,n}\int_{\frac{\pi }{n+1}}^{\pi }\frac{\left\vert \psi _{x}\left(
t\right) \right\vert }{t^{2}}dt
\end{equation*}%
\begin{equation*}
\leq \pi \int_{\frac{\pi }{n+1}}^{\pi }\frac{\left\vert \psi _{x}\left(
t\right) \right\vert }{t^{2}}\sum\limits_{r=\tau }^{n}\frac{a_{n,r}}{r+1}%
dt+\pi a_{n,n}\int_{\frac{\pi }{n+1}}^{\pi }\frac{\left\vert \psi _{x}\left(
t\right) \right\vert }{t^{2}}dt
\end{equation*}%
\begin{equation*}
\leq \pi \sum\limits_{s=1}^{n}\int_{\frac{\pi }{s+1}}^{\frac{\pi }{s}}\frac{%
\left\vert \psi _{x}\left( t\right) \right\vert }{t^{2}}\sum\limits_{r=\tau
}^{n}\frac{a_{n,r}}{r+1}dt+\frac{\pi }{n+1}\int_{\frac{\pi }{n+1}}^{\pi }%
\frac{\left\vert \psi _{x}\left( t\right) \right\vert }{t^{2}}dt.
\end{equation*}%
Since the sequence $\left( \sum\limits_{r=k}^{n}\frac{a_{n,r}}{r+1}\right) $
is nonincreasing in $k,$ after changing the order of summation, we have 
\begin{equation*}
I_{22}\leq \pi \sum\limits_{s=1}^{n}\sum\limits_{r=s}^{n}\frac{a_{n,r}}{r+1}%
\int_{\frac{\pi }{s+1}}^{\frac{\pi }{s}}\frac{\left\vert \psi _{x}\left(
t\right) \right\vert }{t^{2}}dt+\frac{\pi }{n+1}\int_{\frac{\pi }{n+1}}^{\pi
}\frac{\left\vert \psi _{x}\left( t\right) \right\vert }{t^{2}}dt
\end{equation*}%
\begin{equation*}
\leq \pi \sum\limits_{s=1}^{n}\sum\limits_{r=s}^{n}\frac{a_{n,r}}{r+1}%
\left\{ \left[ \frac{1}{t^{2}}\int_{0}^{t}\left\vert \psi _{x}\left(
u\right) \right\vert du\right] _{t=\frac{\pi }{s+1}}^{\frac{\pi }{s}}+2\int_{%
\frac{\pi }{s+1}}^{\frac{\pi }{s}}\left[ \frac{1}{t^{3}}\int_{0}^{t}\left%
\vert \psi _{x}\left( u\right) \right\vert du\right] dt\right\} 
\end{equation*}%
\begin{equation*}
+\frac{\pi }{n+1}\left[ \frac{\widetilde{w}_{x}f\left( \pi \right) }{\pi }%
+2\int_{\frac{\pi }{n+1}}^{\pi }\left[ \frac{1}{t^{3}}\int_{0}^{t}\left\vert
\psi _{x}\left( u\right) \right\vert du\right] dt\right] 
\end{equation*}%
\begin{equation*}
\ll \sum\limits_{s=1}^{n}\sum\limits_{r=s}^{n}\frac{a_{n,r}}{r+1}\left\{ %
\left[ \frac{\widetilde{w}_{x}f\left( \frac{\pi }{s}\right) }{\frac{\pi }{s}}%
-\frac{\widetilde{w}_{x}f\left( \frac{\pi }{s+1}\right) }{\frac{\pi }{s+1}}%
\right] +\widetilde{w}_{x}f\left( \frac{\pi }{s}\right) \int_{\frac{\pi }{s+1%
}}^{\frac{\pi }{s}}\frac{1}{t^{2}}dt\right\} 
\end{equation*}%
\begin{equation*}
+\frac{1}{n+1}\sum\limits_{s=1}^{n}\widetilde{w}_{x}f\left( \frac{\pi }{s+1}%
\right) 
\end{equation*}%
\begin{equation*}
\ll \sum\limits_{r=1}^{n}\frac{a_{n,r}}{r+1}\sum\limits_{s=1}^{r}\left\{ %
\left[ \frac{\widetilde{w}_{x}f\left( \frac{\pi }{s}\right) }{\frac{\pi }{s}}%
-\frac{\widetilde{w}_{x}f\left( \frac{\pi }{s+1}\right) }{\frac{\pi }{s+1}}%
\right] +\widetilde{w}_{x}f\left( \frac{\pi }{s}\right) \right\} 
\end{equation*}%
\begin{equation*}
+\frac{1}{n+1}\sum\limits_{s=1}^{n}\widetilde{w}_{x}f\left( \frac{\pi }{s+1}%
\right) 
\end{equation*}%
\begin{equation*}
=\sum\limits_{r=1}^{n}\frac{a_{n,r}}{r+1}\left\{ \left[ \frac{\widetilde{w}%
_{x}f\left( \pi \right) }{\pi }-\frac{\widetilde{w}_{x}f\left( \frac{\pi }{%
r+1}\right) }{\frac{\pi }{r+1}}\right] +\sum\limits_{s=1}^{r}\widetilde{w}%
_{x}f\left( \frac{\pi }{s}\right) \right\} 
\end{equation*}%
\begin{equation*}
+\frac{1}{n+1}\sum\limits_{s=1}^{n}\widetilde{w}_{x}f\left( \frac{\pi }{s+1}%
\right) 
\end{equation*}%
\begin{equation*}
\leq \sum\limits_{r=1}^{n}\frac{a_{n,r}}{r+1}\left\{ \frac{\widetilde{w}%
_{x}f\left( \pi \right) }{\pi }+\sum\limits_{s=1}^{r}\widetilde{w}%
_{x}f\left( \frac{\pi }{s}\right) \right\} +\frac{1}{n+1}\sum%
\limits_{s=1}^{n}\widetilde{w}_{x}f\left( \frac{\pi }{s+1}\right) 
\end{equation*}%
\begin{equation*}
\ll \sum\limits_{r=1}^{n}\frac{a_{n,r}}{r+1}\sum\limits_{s=0}^{r}\widetilde{w%
}_{x}f\left( \frac{\pi }{s+1}\right) +\frac{1}{n+1}\sum\limits_{s=0}^{n}%
\widetilde{w}_{x}f\left( \frac{\pi }{s+1}\right) 
\end{equation*}%
\begin{equation*}
\ll \sum\limits_{r=0}^{n}a_{n,r}\frac{1}{r+1}\sum\limits_{s=0}^{r}\widetilde{%
\bar{w}}_{x}f\left( \frac{\pi }{s+1}\right) 
\end{equation*}%
and 
\begin{equation*}
I_{23}\leq \int_{\frac{\pi }{n+1}}^{\pi }\frac{\left\vert \psi _{x}\left(
t\right) \right\vert }{t}\left\vert \sum\limits_{r=\tau
}^{n}\sum\limits_{k=\tau }^{r}a_{n,r}b_{r,r-k}\cos \frac{\left(
2r-2k+1\right) t}{2}\right\vert dt
\end{equation*}%
\begin{equation*}
=\int_{\frac{\pi }{n+1}}^{\pi }\frac{\left\vert \psi _{x}\left( t\right)
\right\vert }{t}\left\vert \sum\limits_{k=\tau
}^{n}\sum\limits_{r=k}^{n}a_{n,r}b_{r,r-k}\cos \frac{\left( 2r-2k+1\right) t%
}{2}\right\vert dt
\end{equation*}%
\begin{equation*}
=\int_{\frac{\pi }{n+1}}^{\pi }\frac{\left\vert \psi _{x}\left( t\right)
\right\vert }{t}\left\vert \sum\limits_{k=\tau }^{n}\left[
\sum\limits_{r=k}^{n-1}\left( a_{n,r}b_{r,r-k}-a_{n,r+1}b_{r+1,r+1-k}\right)
\sum_{l=k}^{r}\cos \frac{\left( 2l-2k+1\right) t}{2}\right. \right. 
\end{equation*}%
\begin{equation*}
+\left. \left. a_{n,n}b_{n,n-k}\sum_{l=k}^{n}\cos \frac{\left(
2l-2k+1\right) t}{2}\right] \right\vert dt
\end{equation*}%
\begin{equation*}
\leq \int_{\frac{\pi }{n+1}}^{\pi }\frac{\left\vert \psi _{x}\left( t\right)
\right\vert }{t}\sum\limits_{k=\tau }^{n}\left[ \sum\limits_{r=k}^{n-1}\left%
\vert a_{n,r}b_{r,r-k}-a_{n,r+1}b_{r+1,r+1-k}\right\vert \tau
+a_{n,n}b_{n,n-k}\tau \right] dt
\end{equation*}%
\begin{equation*}
\leq \int_{\frac{\pi }{n+1}}^{\pi }\frac{\left\vert \psi _{x}\left( t\right)
\right\vert }{t}\tau \left[ \sum\limits_{r=\tau }^{n}\sum\limits_{k=\tau
}^{r}\left\vert a_{n,r}b_{r,r-k}-a_{n,r+1}b_{r+1,r+1-k}\right\vert
+a_{n,n}\sum\limits_{k=\tau }^{n}b_{n,n-k}\right] dt
\end{equation*}%
\begin{equation*}
\ll \int_{\frac{\pi }{n+1}}^{\pi }\frac{\left\vert \psi _{x}\left( t\right)
\right\vert }{t^{2}}\left[ \sum\limits_{r=\tau }^{n}\sum\limits_{k=\tau }^{r}%
\frac{a_{n,r}}{\left( r+1\right) ^{2}}+a_{n,n}\sum\limits_{k=0}^{n-\tau
}b_{n,k}\right] dt
\end{equation*}%
\begin{eqnarray*}
&\leq &\int_{\frac{\pi }{n+1}}^{\pi }\frac{\left\vert \psi _{x}\left(
t\right) \right\vert }{t^{2}}\left[ \sum\limits_{r=\tau }^{n}\frac{a_{n,r}}{%
\left( r+1\right) ^{2}}\left( r-\tau +1\right)
+a_{n,n}\sum\limits_{k=0}^{n}b_{n,k}\right] dt \\
&\leq &\int_{\frac{\pi }{n+1}}^{\pi }\frac{\left\vert \psi _{x}\left(
t\right) \right\vert }{t^{2}}\left[ \sum\limits_{r=\tau }^{n}\frac{a_{n,r}}{%
r+1}+a_{n,n}\right] dt.
\end{eqnarray*}%
Further, the same calculation, as that in the estimate of $I_{22}$ , gives
the inequality%
\begin{equation*}
I_{23}\ll \sum\limits_{r=0}^{n}a_{n,r}\frac{1}{r+1}\sum\limits_{s=0}^{r}%
\widetilde{\bar{w}}_{x}f\left( \frac{\pi }{s+1}\right) .
\end{equation*}%
Collecting these estimates we obtain the desired estimate.

Now, we prove the relation (\ref{2.5}). Let%
\begin{equation*}
\widetilde{T}_{n,A,B}\left( x\right) -\widetilde{f}\left( x\right)
=\sum\limits_{r=0}^{n}\sum\limits_{k=0}^{r}a_{n,r}b_{r,k}\widetilde{S}%
_{k}f\left( x\right) -\widetilde{f}\left( x\right)
\end{equation*}%
\begin{eqnarray*}
&=&\frac{1}{\pi }\int_{0}^{\pi }\psi _{x}\left( t\right)
\sum\limits_{r=0}^{n}\sum\limits_{k=0}^{r}a_{n,r}b_{r,k}\widetilde{D}%
_{k}^{\circ }\left( t\right) dt \\
&=&\left( \frac{1}{\pi }\int_{0}^{\frac{\pi }{n+1}}+\frac{1}{\pi }\int_{%
\frac{\pi }{n+1}}^{\pi }\right) \psi _{x}\left( t\right)
\sum\limits_{r=0}^{n}\sum\limits_{k=0}^{r}a_{n,r}b_{r,k}\widetilde{D}%
_{k}^{\circ }\left( t\right) dt \\
&=&J_{1}+J_{2}.
\end{eqnarray*}%
Further%
\begin{eqnarray*}
\left\vert J_{1}\right\vert &\leq &\frac{1}{\pi }\int_{0}^{\frac{\pi }{n+1}%
}\left\vert \psi _{x}\left( t\right) \right\vert \left\vert
\sum\limits_{r=0}^{n}\sum\limits_{k=0}^{r}a_{n,r}b_{r,k}\frac{\cos \frac{%
\left( 2k+1\right) t}{2}}{2\sin \frac{t}{2}}\right\vert dt \\
&\leq &\frac{1}{2}\int_{0}^{\frac{\pi }{n+1}}\frac{\left\vert \psi
_{x}\left( t\right) \right\vert }{t}\left\vert
\sum\limits_{r=0}^{n}\sum\limits_{k=0}^{r}a_{n,r}b_{r,k}\cos \frac{\left(
2k+1\right) t}{2}\right\vert dt \\
&\leq &\frac{1}{2}\int_{0}^{\frac{\pi }{n+1}}\frac{\left\vert \psi
_{x}\left( t\right) \right\vert }{t}dt\ll \widetilde{w}_{x}f\left( \frac{\pi 
}{n+1}\right)
\end{eqnarray*}

\begin{equation*}
=\widetilde{w}_{x}f\left( \frac{\pi }{n+1}\right)
\sum\limits_{r=0}^{n}a_{n,r}\leq \sum\limits_{r=0}^{n}a_{n,r}\widetilde{\bar{%
w}}_{x}f\left( \frac{\pi }{r+1}\right)
\end{equation*}%
\begin{equation*}
\ll \sum\limits_{r=0}^{n}a_{n,r}\left[ \frac{1}{r+1}\sum\limits_{k=0}^{r}%
\widetilde{\bar{w}}_{x}f\left( \frac{\pi }{k+1}\right) \right]
\end{equation*}%
and as above%
\begin{equation*}
\left\vert J_{2}\right\vert \ll \sum\limits_{r=0}^{n}a_{n,r}\left[ \frac{1}{%
r+1}\sum\limits_{k=0}^{r}\widetilde{\bar{w}}_{x}f\left( \frac{\pi }{k+1}%
\right) \right] .
\end{equation*}

$\blacksquare $

\subsection{Proof of Theorems 2}

Theorem 2 is special case of Theorem 1, therefore, analogously as in \cite[%
Proof of Theorem 2]{LSz}, we can immediately obtain our estimates.$%
\blacksquare $

\subsection{Proofs of Theorems 3 and 4}

The proofs are similar to the above. The results follow by (\ref{81}) and (%
\ref{2.6}). $\blacksquare $

\end{document}